\documentclass[10pt]{article}
 
\usepackage{amsmath, amssymb, amsthm, amsfonts, graphicx} 

\usepackage{cite, color}
\usepackage[normalem]{ulem}

\usepackage{color}

\definecolor{orange}{rgb}{1,0.5,0}

\newtheorem{theorem}{Theorem}[section]

\newtheorem{proposition}[theorem]{Proposition}

\newtheorem{lemma}[theorem]{Lemma}
\newtheorem{corollary}[theorem]{Corollary}

\newcommand{\lam}{\lambda}
\newcommand{\bet}{\beta}
\newcommand{\gam}{\gamma}
\newcommand{\eps}{\varepsilon}
\newcommand{\alp}{\alpha}

\newcommand{\Ome}{\Omega}

\newcommand{\HH}{{\mathcal H}}

\DeclareMathOperator*{\diam}{diam}

\newcommand{\R}{\mathbb{R}}

\renewcommand{\phi}{\varphi}

\newcommand{\rring}[1]%
{{\mathop{#1}\limits^{\raisebox{-0.2ex}[0ex][0ex]{\tiny{oo}}}}}

\def\XXint#1#2#3{{\setbox0=\hbox{$#1{#2#3}{\int}$}
\vcenter{\hbox{$#2#3$}}\kern-.5\wd0}}

\newcommand{\upref}[2]{\hspace{-0.8ex}\stackrel{\eqref{#1}}{#2}}

\newcommand{\lupref}[2]{\hspace{0ex} \stackrel{\eqref{#1}}{#2}}

\newcommand{\ncoL}[2]{\|#1\|_{L^\infty({#2})}}

\newcommand{\nlt}[1]{\|#1\|_{L^2}}
\newcommand{\nltL}[2]{\|#1\|_{L^2({#2})}}

\newcommand{\NL}{V}

\numberwithin{equation}{section}
\begin{document}

\title{\bf On an isoperimetric problem with a competing non-local
  term. II. The general case.}

\author{Hans Kn\"upfer\thanks{Hausdorff Centre for Mathematics, University of
    Bonn, 53117 Bonn, Germany} \and Cyrill B. Muratov\thanks {Department of
    Mathematical Sciences, New Jersey Institute of Technology, Newark, NJ
    07102}}

\maketitle

\begin{abstract}
  This paper is the continuation of [H. Kn\"upfer and C. B. Muratov,
  Commun. Pure Appl. Math. (2012, to be published)].  We investigate the
  classical isoperimetric problem modified by an addition of a non-local
  repulsive term generated by a kernel given by an inverse power of the
  distance. In this work, we treat the case of general space dimension. We
  obtain basic existence results for minimizers with sufficiently small masses.
  For certain ranges of the exponent in the kernel we also obtain non-existence
  results for sufficiently large masses, as well as a characterization of
  minimizers as balls for sufficiently small masses and low spatial
  dimensionality. The physically important special case of three space
  dimensions and Coulombic repulsion is included in all the results mentioned
  above. In particular, our work yields a negative answer to the question if
  stable atomic nuclei at arbitrarily high atomic numbers can exist in the
  framework of the classical liquid drop model of nuclear matter. In all cases
  the minimal energy scales linearly with mass for large masses, even if the
  infimum of energy may not be attained.
\end{abstract}


\section{Introduction}

This paper is the second part of \cite{KnuepferMuratov-2011b}, in
which a non-local modification of the classical isoperimetric problem
was considered. Namely, we wish to examine minimizers of the energy
functional
\begin{align} \label{EE} %
  E(u) = \int_{\mathbb R^n} |\nabla u| \, dx %
  + \int_{\R^n} \int_{\R^n} \frac{u(x) u(y)}{|x - y|^\alpha} \ dx dy,
\end{align}
in which $u \in BV(\mathbb R^n;\{ 0, 1\})$, $n \geq 2$ and $\alpha \in
(0,n)$. We assume that the mass associated with $u$ is prescribed,
i.e.
\begin{align}
  \label{mass}
  \int_{\R^n} u \ dx \ = \ m,
\end{align}
for some $m \in (0,\infty)$. Note that the considered range of values
of $\alpha \in (0,n)$ ensures that the non-local part of the energy in
\eqref{EE} is always well-defined.

\medskip

The above problem arises in a number of physical contexts
\cite{Muratov-1998}. A case which is of particular physical importance
is the one of $n = 3$ and $\alpha = 1$, corresponding to {\em
  Coulombic repulsion} (for an overview, see
\cite{Muratov-2002}). Perhaps the earliest example where the model in
\eqref{EE} and \eqref{mass} appears is the liquid drop model of the
atomic nuclei proposed by Gamow in 1928 \cite{gamow30} and then
developed by von Weizs\"acker \cite{weizsacker35}, Bohr
\cite{bohr36,bohr39} and many other researchers later on. This model
was used to explain various properties of nuclear matter and, in
particular, the mechanism of nuclear fission (for more recent studies,
see e.g.  \cite{cohen62,cohen74,pelekasis90,myers96}).  Due to the
fundamental nature of Coulombic interaction, the same model (or its
diffuse interface analog) also arises in many other physical
situations (see
e.g. \cite{CareMarch-1975,EmeryKivelson-1993,Nagaev-1995,
  ChenKhachaturyan-1993}) and, in particular, is relevant to a variety
of polymer, as well as other systems (see e.g.
\cite{NyrkovaKhokhlovEtAl-1994,
  OhtaKawasaki-1986,Degennes-1979,Stillinger-1983,GlotzerDiEtAl-1995,
  KovalenkoNagaev-1986,Mamin-1994}).

\medskip

It is well known that the local part of the energy in \eqref{EE},
which leads to the classical isoperimetric problem, is uniquely
minimized by balls among all sets of finite perimeter with prescribed
mass \cite{Degiorgi-1958}. The key ingredient in the proof of this
celebrated result by De Giorgi is the use of Steiner symmetrization,
which lowers the interfacial energy. The effect of the rearrangement
in the Steiner symmetrization, however, is quite different for the
non-local part of the energy in \eqref{EE}. Since by the rearrangement
the mass is transported closer together, the resulting non-local
energy actually increases. It is this competition of the cohesive
forces due to surface tension and the repulsive long-range forces that
makes this variational problem highly non-trivial. In particular,
minimizers are no longer expected to be convex or even exist at all
for certain ranges of the parameters.  To take the particular case of
the nuclear drop model, we are not aware of any prior studies
establishing existence or non-existence of minimizers for large masses
which would not assume spherical symmetry of the drop.  An {\em
  ansatz-free} answer to this question, however, is essential in order
to develop a basic understanding of the properties of atoms.

\medskip

In our previous work \cite{KnuepferMuratov-2011b}, we have
investigated the two-dimensional version of the variational problem
associated with \eqref{EE}. We investigated existence, non-existence
and shape of minimizers in the full range of $\alpha \in (0,2)$ for $n
= 2$ and provided a complete characterization of the minimizers for
sufficiently small $\alpha$. The special topological structure of
$\R^2$ simplifies many of the arguments used in that work.  In the
present work, we extend these results to the general case of $n \geq
2$ space dimensions. In particular, we investigate existence, shape
and regularity of minimizers for prescribed mass $m \in
(0,\infty)$. Due to the technical difficulties associated with the
transition from the $n = 2$ case to $n \geq 3$, however, we are able
to treat in a similar way only certain ranges of $\alpha$ and $n$. In
particular, we need to work within the general framework of sets of
finite perimeter.

\medskip

As in \cite{KnuepferMuratov-2011b}, we are able to establish existence
of minimizers for sufficiently small masses for all $n \geq 3$ and the
full range of $\alpha \in (0, n)$. At the same time, we are only able
to prove that balls are the unique minimizers (up to translations) for
sufficiently small masses when $n \leq 7$ and $\alpha \in (0,
n-1)$. Note that the physically relevant special case of $n = 3$ and
$\alpha = 1$ is included. The first restriction, $n \leq 7$, seems to
be of technical nature: In fact, for our arguments we use the result
that for $n \leq 7$ quasiminimizers of the perimeter have smooth
boundaries (see, e.g., \cite{ambrosio98,tamanini84}). Perhaps one
could remove this restriction by using more sophisticated machinery of
the regularity theory for quasiminimizers, e.g., following the ideas
of the recent work by Figalli and Maggi \cite{figalli11}. The second
restriction, $\alpha \in (0,n-1)$, however is of more fundamental
nature and distinguishes the case of far field-dominated regime
$\alpha < n - 1$ from the near field-dominated regime $\alpha \geq n -
1$ (cf. also with \cite{KnuepferMuratov-2011b}). In the latter case
the potential associated with the minimizer is no longer
Lipschitz-continuous (cf. \eqref{v}, \eqref{v-est}). Hence, a
different approach is needed in this case to deal with the shape of
minimizers; this is necessary even in the regime, in which the
perimeter term dominates the non-local term.

\medskip

Similarly, we are only able to prove non-existence of minimizers for
large masses in the case $\alpha < 2$. Observe that, once again, our
result covers the physical most relevant case of Coulombic
interaction, i.e., $n = 3$ and $\alpha = 1$. We note that for the
latter case a non-existence proof was also very recently obtained by
Lu and Otto in their study of the Thomas-Fermi-Dirac-von Weizs\"acker
model of quantum electron gas, using different arguments
\cite{lu12}. From the point of view of applications, our result
provides a basic non-existence result for uniformly charged drops with
sufficiently large masses minimizing the energy in \eqref{EE} and, in
particular, for ground states of atomic nuclei with sufficiently large
atomic numbers within the charged drop model of nuclear matter. Our
analysis also partially substantiates the picture described in
\cite{ChoksiPeletier-2010,ChoksiPeletier-2011} for the Coulombic case
in three space dimensions. The non-existence proof fails in the
opposite case of $\alpha \geq 2$, and we do not know whether the
result still holds for some $\alpha \in [2, n)$, when the potential
has shorter range. What we did show is that independently of $\alpha$
the minimal energy always scales linearly with mass for large
masses. Note that this result is consistent both with a minimizing
sequence consisting of many isolated balls moving away from each other
and with a minimizer in the form of a long ``sausage-shaped''
drop. Which of these two alternatives occurs for the large mass case
and $\alpha \geq 2$ remains to be studied.

\medskip

Our paper is organized as follows. In Sec. \ref{s-not}, we introduce
the basic notions of the geometric measure theory. Here we also
reformulate our variational problem in the framework of sets of finite
perimeter, provide some technical results that will be used in the
analysis and describe all the notations.  In Sec. \ref{s-main}, we
state the main results of the paper and outline the key ideas of their
proofs. In Sec. \ref{s-bas}, we prove several technical lemmas that
are used throughout the rest of the paper. In Sec. \ref{s-exist}, we
establish existence of minimizers (Theorem \ref{thm-existence}) for
small masses. In Sec. \ref{s-ball}, we establish the precise shape of
the minimizers for small masses in a certain range of the parameters
(Theorem \ref{thm-ball}). Finally, in Sec. \ref{s-non} we establish
non-existence of minimizers for large masses in a certain range of the
parameters (Theorem \ref{thm-nonexistence}) and the scaling and
equipartition of energy for large masses in the whole range of the
parameters (Theorem \ref{thm-scaling}).

\section{Notation and sets of finite perimeter}
\label{s-not}

The variational problem \eqref{EE}-\eqref{mass} is most conveniently
addressed in the setting of geometric measure theory. In this section,
we first introduce some basic measure theoretic notions; here we refer
to \cite{AmbrosioFuscoEtal-Book,AmbrosioCasellesEtal-2001} as
references. We then reformulate \eqref{EE} in terms of sets of finite
perimeter. We conclude the section by recalling some basic regularity
results for minimizers.

\paragraph{Some measure theoretic notions:} 
We say that a function $u \in L^1(\R^n)$ has bounded variation, $u \in
BV(\R^n)$, if
\begin{align}
  \int_{\R^n} |\nabla u| \, dx := \sup_{\| \zeta \|_\infty \leq 1}
  \left\{ \int_{\R^n} u \nabla \cdot \zeta \, dx \ : \ \zeta \in
    C^1_c(\R^n; \R^n) \right\} < \infty.
\end{align}
For any measurable set $F \subset \R^n$, we denote by $|F|$ its
$n$-dimensional Lebesgue measure.  Moreover, $F$ is said to have
finite parameter if $\chi_F \in BV(\R^n)$, where $\chi_F$ is the
characteristic function of $F$; its perimeter is then defined by $P(F)
:= \int_{\R^n} |\nabla \chi_F| dx$. The $k$-dimensional Hausdorff
measure with $k \in [0,n]$ is denoted by $\HH^k(F)$. We will
frequently use Fubini's theorem for measures \cite[Theorem
1.74]{AmbrosioFuscoEtal-Book}, as well as the co-area formula for
integration in spherical coordinates \cite[Proposition 1 of
Sec. 3.4.4]{EvansGariepy-Book}.

\medskip

For any Lebesgue measurable set $F$, its upper density at a point $x
\in \R^n$ is
\begin{align}
  \overline D(F,x) \ := \ \limsup_{r \to 0} \frac{|F \cap
    B_r(x)|}{|B_r(x)|},
\end{align} 
where $B_r(x)$ is the open ball with center $x$ and radius $r$. The
{\em essential interior} $\mathring{F}^M$ of $F$ is then defined as
the set of all $x \in \mathbb R^n$ for which $\overline{D}(F, x) =1$,
while the {\em essential closure} $\overline{F}^M$ of $F$ is defined
as the set of all $x \in \mathbb R^n$ for which $\overline{D}(F, x) >
0$.  The {\em essential boundary} $\partial^M F$ of $F$ is defined as
the set of all points where $\overline D(F,x) > 0$ and $\overline
D(\R^n \backslash F,x) > 0$. By a result of Federer, a set has finite
perimeter if and only if $\HH^{n-1}(\partial^M F) < \infty$. The {\em
  reduced boundary} $\partial^* F$ of a set of finite perimeter $F$ is
defined as a set of all points $x \in
\partial^M F$ such that the measure-theoretic normal exists at $x$,
i.e., if the following limit exists:
\begin{align}
  \label{eq:FE}
  \nu_F(x) \ := \ \lim_{r \to 0} {\int_{B_r(x)} \nabla \chi_F(y) dy \over
    \int_{B_r(x)} |\nabla \chi_F(y)| dy} \quad \text{and} \quad
  |\nu_F(x)| = 1,
\end{align}
where $\nabla \chi_F$ denotes the vector-valued Radon measure associated with
the distributional derivative of $\chi_F$ and $|\nabla \chi_F|$ coincides with
the $\mathcal H^{n-1}$ measure restricted to $\partial^M F$.  Again, by a result
of Federer we have $\mathcal H^{n-1}(\partial^M F \backslash \partial^* F) = 0$
\cite[Theorem 3.61]{AmbrosioFuscoEtal-Book}.

\medskip

Note that the topological notion of connectedness is not well-defined
for sets of finite perimeter, since these sets are only defined up to
$\mathcal H^n$-negligible sets. However, the following generalization
of the notion of connectedness can be defined: We say that a set $F$
of finite perimeter is {\it decomposable} if there exists a partition
$(A,B)$ of $F$ such that $P(F) = P(A) + P(B)$ for two sets $A,B$ with
positive Lebesgue measure. Otherwise the set is called {\em
  indecomposable}, which is the measure theoretic equivalent of the
notion of a connected set. Similarly, we say that a measurable set $F$
is {\it essentially bounded} if its essential closure $\overline{F}^M$
is bounded.

\paragraph{Notations for the isoperimetric problem:} The isoperimetric
deficit of a set of finite perimeter $F \subset \R^n$ is defined in
this paper by
\begin{align}
  \label{isodef}
  D(F) \ := \ \frac{P(F)}{n \omega_n^{1 \over n} |F|^{n-1 \over n }} -
  1,
\end{align}
where $\omega_n = \frac{\pi^{n/2}}{\Gamma(\frac n2 + 1)}$ denotes the measure of
the unit ball in $\R^n$. 

\medskip

A natural notion of the difference of two measurable sets $F$ and $G$
with $|F| = |G|$ is the Fraenkel asymmetry:
\begin{align}
  \label{Fasym}
  \Delta(F,G) \ := \ \min_{x \in \R^n} { |F \triangle (G + x)| \over
    |F|},
\end{align}
where $F \triangle G := (F \backslash G) \cup (G \backslash F)$
denotes the symmetric difference of the sets $F$ and $G$. The
following quantitative version of the isoperimetric inequality
relating the Fraenkel asymmetry \eqref{Fasym} and the isoperimetric
deficit \eqref{isodef} was recently established
\cite{FuscoMaggiPratelli-2008}:
\begin{align}
  \label{qiso}
  \Delta (F, B) \ \leq \ C_n \sqrt{D(F)},
\end{align}
where $B$ is a ball with $|B| = |F|$ and $C_n$ is a positive constant
that depends only on the dimension $n$.

\medskip

Another important notion is the notion of {\it quasiminizer of the
  perimeter} (see, e.g., \cite{ambrosio98,tamanini84}).  A set $F$ of
finite perimeter is called a quasiminimizer of the perimeter (with
prescribed mass), if there exist a constant $C > 0$ such that for all
$G \subset \R^n$ with $|G| = |F|$ and $F \triangle G \subset B_r(0)$
for some and $r > 0$, one has
\begin{align}
  \label{qm}
  P(F) \ \leq \ P(G) + C |F \triangle G|.
\end{align}
As will be shown below (see Proposition \ref{thm-regularity}),
minimizers of our variational problem are quasiminimizers of the
perimeter in the above sense and, therefore, the regularity results of
\cite{Rigot-2000,xia05} apply to them.

\paragraph{The variational model:} We express \eqref{EE} as a
functional on sets of finite perimeter. For any measurable set $F
\subset \R^n$, let the potential $v_F$ be given by
\begin{align}
  \label{v}
  v_F (x) \ := \ \int_F \frac 1{|x-y|^\alp} \ dx.
\end{align}
The nonlocal part of the energy in \eqref{EE} can then be expressed as
\begin{align}
  \NL(F) \ := \ \int_F v_F \ dx \ = \ \int_{F} \int_{F} \frac 1{|x -
    y|^\alpha} \ dx dy.
\end{align}
The energy \eqref{EE} can hence be expressed as
\begin{align} \label{E} %
  E(F) \ = \ P(F) + \NL(F).
\end{align}
We say that $\Omega$ is a minimizer of \eqref{E} if $E(\Omega) \leq
E(F)$ for all sets of finite parameter $F$ with $|F| = |\Ome|$. In the
following, we will reserve the symbol $\Omega$ to denote
minimizers. 
It can be shown that minimizers of the energy are solutions of an
Euler-Lagrange equation in a suitable sense. In this paper, however,
we will not use the Euler-Lagrange equation of this variational
problem, but instead we will use only the energy to obtain the
necessary estimates. In this sense, the methods used in this work are
more general than the approach in our related work for the
two-dimensional case where we used the Euler-Lagrange equation
\cite{KnuepferMuratov-2011b}.

\medskip

We have the following general result concerning the regularity of
minimizers of the considered variational problem:

\begin{proposition} \label{thm-regularity} %
  Let $\Omega$ be a minimizer for \eqref{E}. Then the reduced boundary
  $\partial^* \Omega$ of $\Omega$ is a $C^{1,\frac12}$ manifold.
  Furthermore, $\HH^k(\partial^M \Omega \backslash \partial^* \Omega)
  = 0$ for all $k > n - 8$. In particular, for $n \leq 7$ the set
  $\Omega$ is (up to a negligible set) open with boundary of class
  $C^{1,\frac12}$. The complement of $\Omega$ has finitely many
  connected components.
\end{proposition} 
\begin{proof}
  The proof is an adaptation of the results of \cite{Rigot-2000} and
  \cite{xia05}.  In \cite{Rigot-2000}, Rigot established the regularity of a
  class of quasiminimizers of the perimeter with prescribed mass, which includes
  our notion \eqref{qm} of quasiminimizers. It is hence enough to show that
  every minimizer $\Ome$ of \eqref{E} is also a quasiminimizer in the sense of
  definition \eqref{qm}. The statement of the Theorem then follows from
  \cite[Theorem 1.4.9]{Rigot-2000} and \cite[Theorem 4.5]{xia05}. Let $\Omega$
  be a minimizer of \eqref{E} with prescribed mass and let $F$ be a set of
  finite perimeter with $|F| = |\Omega|$ and $F \triangle \Omega \subset B_r(0)$
  for some $r > 0$. By the minimizing property of $\Ome$, we have
  \begin{align}
    \label{Equasimin2}
    P(\Omega) - P(F) \ %
    & \leq V(F) - V(\Omega) \ %
    \lupref{v}\leq \ \int_{\Omega
      \triangle F} ( v_\Omega + v_F ) \, dx \notag \\
    &\leq 2 \, |\Omega \triangle F| \bigg( \int_{B_1(0)} {1 \over
      |y|^\alpha} \, dy + m \bigg) \leq C \, |\Omega \triangle F|,
  \end{align}
  for some $C > 0$ depending only on $n$, $\alpha$ and $m$. It follows
  that the minimizers of \eqref{E} are also quasiminimizers of the
  perimeter.
\end{proof}

\paragraph{Other notations:}

Unless otherwise noted, all constants throughout the proofs are
assumed to depend only on $n$ and $\alpha$. The symbol $e_k$ is
reserved for the unit vector in the $k$-th coordinate direction.

\section{Main results}
\label{s-main}

In our first result we show that for sufficiently small masses there exists a
minimizer of the considered variational problem.

\begin{theorem}[Existence of minimizers] \label{thm-existence} %
  For all $n \geq 3$ and for all $\alp \in (0,n)$ there is a mass $m_1
  = m_1(\alpha, n) > 0$ such that for all $m \leq m_1$, the energy in
  \eqref{E} has a minimizer $\Omega \subset \R^n$ with $|\Omega| =
  m$. The minimizer $\Omega$ is essentially bounded and
  indecomposable.
\end{theorem}
The proof of this theorem follows by the direct method of calculus of
variations, once we show that for sufficiently small mass every
minimizing sequence of the energy may be replaced by another
minimizing sequence where all sets have uniformly bounded essential
diameter.  By the regularity result in Proposition
\ref{thm-regularity}, we also obtain certain regularity of the
minimizer's boundary. We note that our result improves an existence
result of \cite{ChoksiPeletier-2010,ChoksiPeletier-2011} for the
Coulombic case $n = 3$ and $\alpha = 1$, demonstrating that there
exists an {\em interval} of masses near the origin for which the
minimizers indeed exist.

\medskip

By analogy with the two-dimensional case, it is natural to expect that
if the mass is sufficiently small, the minimizer of the considered
variational problem is precisely a ball \cite{KnuepferMuratov-2011b}.
Our next result shows that this is indeed the case at least in a
certain range of values for $\alpha$.
\begin{theorem}[Ball is the minimizer] \label{thm-ball} %
  For all $3 \leq n \leq 7$ and for all $\alp \in (0,n-1)$ there is a
  mass $m_0 = m_0(\alpha, n) > 0$ such that for all $m \leq m_0$, the
  unique (up to translation) minimizer $\Omega \subset \R^n$ of
  \eqref{EE} with $|\Omega| = m$ is given by a ball.
\end{theorem}
For the proof of Theorem \ref{thm-ball}, which applies to the regime
where the perimeter is the dominant term in the energy, we make use
both of the quantitative isoperimetric inequality in \eqref{qiso} and
regularity estimates for the minimizer as given by Rigot
\cite{Rigot-2000}. We first show that the minimizer is close to a ball
in the $C^1$-sense (cf. with \cite{Fuglede-1989}). By the Lipschitz
continuity of the nonlocal potential, we then deduce that the minimum
energy can in fact be only achieved by precisely a ball. Let us note
that the assumption $n \leq 7$ in Theorem \ref{thm-ball} seems to be
of only technical nature and it may be possible to adapt recent
results of \cite{figalli11} to extend the statement of Theorem
\ref{thm-ball} to all $n \geq 8$ and $\alpha < n - 1$. However, our
method of proof for the $\alpha < n - 1$ case does not extend
straightforwardly to the case of $\alpha \geq n - 1$. Indeed, the
technical difficulties encountered in the latter case become rather
substantial even for $n = 2$ \cite{KnuepferMuratov-2011b}. Let us
mention that some related recent results for the $n$-dimensional
Coulombic case on bounded domains were obtained in
\cite{CicaleseSpadaro-2011}.

\medskip

On the contrary, for large masses the repulsive interaction dominates
and the variational problem does not admit a minimizer:
\begin{theorem}[Non-existence of minimizers]
  \label{thm-nonexistence} %
  For all $n \geq 3$ and for all $\alp \in (0,2)$, there is $m_2 =
  m_2(\alpha, n)$ such that for all $m \geq m_2$, the energy in
  \eqref{EE} does not admit a minimizer $\Omega \subset \R^n$ with
  $|\Omega| = m$.
\end{theorem}

For the proof, we first show that minimizers must be indecomposable.
On the other hand, we can cut any indecomposable set with sufficiently
large mass by a hyperplane into two large pieces. We move the two
pieces far apart from each other and compare the energy of the new set
with the original configuration. The resulting inequality can be
expressed as a differential inequality for the mass of cross-sections
of the minimizer by different hyperplanes. The proof is concluded by a
contradiction argument, whereby the resulting estimate of the total
mass is too high.

\medskip

Our result is restricted to the case $\alpha \in (0,2)$.  Note that
the above result does, in particular, apply to the physically
important special case of Coulomb interaction, i.e. $n = 3$ and
$\alpha = 1$ (see also \cite{lu12}). It is an interesting open
question if the non-existence result extends to arbitrary $\alpha \in
[2,n)$. In particular, for $\alp \to n$ the nonlocal energy is
dominated by short-range interactions.

  \medskip

\begin{theorem}[Scaling and equipartition of
  energy] \label{thm-scaling} %
  For all $n \geq 3$ and for all $\alpha \in (0,n)$ there exist two
  constants $C,c > 0$ only depending on $n$ and $\alpha$ such that for
  the energy in \eqref{EE} we have
  \begin{align} \label{EcC}%
    c \max \{ m^{\frac{n-1}n}, m \} \ \leq \ \inf_{|\Omega| = m} E(\Omega) \
    \leq \ C \max \{ m^{\frac{n-1}n}, m \}.
  \end{align}
  Furthermore, for $m \geq 1$ we have equipartition of energy, in the sense that
  for every set of finite perimeter $\Omega \subset \mathbb R^n$ satisfying
  $|\Omega| = m$ and $E(\Omega) \leq \beta m$ with some $\beta > 0$ we have
  \begin{align}
    \label{eq:1}
    c_\bet m \ %
    \leq \ %
    \min \{ P(\Omega) , V(\Omega) \} \ %
    \leq \ \max \{ P(\Omega) , V(\Omega) \} \ %
    \leq \ \beta m,
  \end{align}
  for some $c_\beta > 0$ only depending on $\alpha$, $n$ and $\beta$,
  but not on $m$.
\end{theorem}

At the core of the proof of this theorem is the proof of the lower
bound of the energy. This estimate follows from an interpolation
inequality which connects interfacial and nonlocal parts of the
energy.

\section{Some basic estimates}
\label{s-bas}

We start our analysis with a few auxiliary lemmas that will be useful
in what follows. By a simple argument using the regularity result in
Proposition \ref{thm-regularity}, it follows that minimizers are
essentially bounded and indecomposable:

\begin{lemma}[Boundedness and connectedness of minimizers]
  \label{lem-indec}
  Let $\Omega \subset \mathbb R^n$ be a minimizer of \eqref{E} with
  $|\Omega| = m$. Then $\Omega$ is essentially bounded and
  indecomposable.
\end{lemma}
\begin{proof}
  As was shown in the proof of Proposition \ref{thm-regularity}, we
  can apply \cite[Lemma 2.1.3]{Rigot-2000} guaranteeing that there
  exists $r > 0$ and $c > 0$ such that for every $x \in
  \overline{\Omega}^M$ we have $|\Omega \cap B_r(x)| \geq c r^n$. If
  $\Omega$ is not essentially bounded, then there exists a sequence
  $(x_k) \in \overline{\Omega}^M$ such that $x_k \to \infty$ and $|x_k
  - x_{k'}| > 2 r$ for all $k, k'$. Then clearly
    \begin{align}
      \label{xkminf}
      |\Omega| \ \geq \  \sum_k |\Omega \cap B_r(x_k)| \ %
      \geq \ \sum_k c r^n = \infty,
    \end{align}
    contradicting the fact that $|\Omega| = m < \infty$.


  \medskip
  
  To prove that the minimizers are indecomposable, suppose the
  opposite is true and that there exist two sets of finite perimeter
  $\Omega_1$ and $\Omega_2$ such that $\Omega_1 \cap \Omega_2 =
  \varnothing$ and $\Omega = \Omega_1 \cup \Omega_2$, with $P(\Omega)
  = P(\Omega_1) + P(\Omega_2)$. Since $\Omega$ and, hence, $\Omega_1$
  and $\Omega_2$ are essentially bounded, defining $\Omega_R :=
  \Omega_1 \cup (\Omega_2 + e_1 R)$, we have $|\Omega_R| = m$ and
  $P(\Omega_R) = P(\Omega)$ for $R > 0$ sufficiently large. At the
  same time, the nonlocal energy decreases:
  \begin{align}
    \label{EOmindec}
    \liminf_{R \to \infty} E(\Omega_R) & = P(\Omega_R) + V(\Omega_1)
    +
    V(\Omega_2) \nonumber \\
    & \hspace{-1cm} < P(\Omega) + V(\Omega_1) + V(\Omega_2) + 2
    \int_{\Omega_1} \int_{\Omega_2} {1 \over |x - y |} \, dx \, dy =
    E(\Omega).
  \end{align}
  Thus, choosing $R$ sufficiently large, we obtain $E(\Omega_R) <
  E(\Omega)$, contradicting the minimizing property of $\Omega$.
\end{proof}
Our next lemma yields a general criterion on a set of finite perimeter
being energetically unfavorable that will be helpful for several of
our proofs.
\begin{lemma}[Non-optimality criterion] \label{lem-split} %
  Let $F \subset \R^n$ be a set of finite perimeter. Suppose there is a
  partition of $F$ into two disjoint sets of finite perimeter $F_1$ and
  $F_2$ with positive measures such that
  \begin{align} \label{est-sigma} %
    \Sigma \ := \ P(F_1) + P(F_2) - P(F) \ \leq \
    \frac12 E(F_2).
  \end{align}
  Then there is $\eps > 0$ depending only on $n$ and $\alp$ such that
  if
  \begin{align} \label{om2eps1m} |F_2| \ \leq \ \eps \min \{ 1,
    |F_1| \},
  \end{align}
  there exists a set $G \subset \R^n$ such that $|G| =
  |F|$ and $E(G) < E(F)$.
\end{lemma}
\begin{proof}
  Let $m_1 := |F_1|$, $m_2 := |F_2|$ and let $\gamma :=
  \frac{m_2}{m_1} \leq \eps$.  We will compare $F$ with the
  following two sets:
  \begin{itemize}
  \item The set $\tilde F$ given by $\tilde F = \ell
    F_1$, where $\ell := \sqrt[n]{1 + \gamma}$. In particular,
    $|\tilde F| = |F|$.
  \item The set $\hat F$ given by a collection of $N \geq 1$ balls of
    equal size and with centers located at $x = jR e_1$, $j = 1,
    \ldots, N$, with $R$ large enough. The number $N$ is chosen to be
    the smallest integer for which the mass of each ball does not
    exceed 1 and such $|\hat F| = |F|$.
  \end{itemize}
  If $E(\hat F) < E(F)$, then the assertion of the lemma
  holds true and the proof is concluded. Therefore, in the following
  we may assume
  \begin{align} \label{ip} %
    E(F) \ \leq \ E(\hat F) \ \leq \ C \max \big \{ m,
    m^{\frac{n-1}n} \big \},
  \end{align}
  for some $C > 0$, where the last inequality is obtained by direct computation.

  \medskip

  It hence remains to show that under the assumption \eqref{ip} and
  for sufficiently small $\eps$, we have $E(\tilde F) < E(F)$. We
  first note that in view of the scaling of interfacial and nonlocal
  energies, we have
  \begin{align}
    \label{Omtildeupper}
    E(\tilde F) \ %
    &= \ \ell^{n-1} P(F_1) + \ell^{2n-\alpha} \NL(F_1)  \notag \\
    &\leq \ E(F_1) + \left( (\ell^{n-1}-1) + (\ell^{2n-\alpha}-1)
    \right) E(F_1).
  \end{align}
  Choosing $\eps \leq 1$, we have $1 \leq \ell \leq (1 + \eps)^{\frac
    1n} \leq 2^{1 \over n}$ and, therefore, by Taylor's formula we
  obtain $\ell^{n-1} - 1 \leq K (\ell - 1)$ and $\ell^{2n-\alpha} -1
  \leq K (\ell - 1)$ for some $K > 0$ independent of
  $\ell$. Furthermore, since $\ell-1 \leq \gamma$ and by
  \eqref{Omtildeupper}, we arrive at $ E(\tilde F) - E(F_1) \leq 2
  \gamma K E(F_1)$. By the definition of $\Sigma$ and with \eqref{E},
  this implies
  \begin{align}
    E(\tilde F) - E(F) \ %
    &\leq \ V(F_1) + V(F_2) - V(F) + \Sigma -
    E(F_2) +  2 \gamma K E(F_1) \nonumber \\
      &\upref{est-sigma}< \ - \frac12 E(F_2) + 2 \gam K E(F_1),
    \label{rhs-of}
  \end{align}
  where for the second estimate, we also used the fact that $V(F_1) +
  V(F_2) < V(F)$. By positivity of $V$ and the isoperimetric
  inequality we have $E(F_2) > P(F_2) \geq c m_2^{\frac{n-1}n}$ for
  some $c > 0$. Furthermore, a straightforward calculation using
  \eqref{est-sigma} and $V(F) > V(F_1) + V(F_2)$ yields
  $E(F_1) < E(F)$, so that \eqref{rhs-of} turns into
  \begin{align}
    E(\tilde F) - E(F) \ %
    &\upref{rhs-of}\leq - c m_2^{n - 1 \over n} + C \gam E(F)
    \notag \\ %
    & \upref{ip} \leq - c m_2^{n - 1 \over n} + C \max\{
    m_2, \eps^{1 \over n} m_2^{n-1 \over n} \},
    \label{deltaEm2}
  \end{align}
  for some $C, c > 0$, where we also used that $\gam m \leq 2 m_2$ and
  that $\gam \leq \eps$ by \eqref{om2eps1m}. Then, since $m_2 \leq
  \eps$ by \eqref{om2eps1m} as well, the assertion of the lemma
  follows for $\eps$ sufficiently small.
\end{proof}

Our next lemma is an improvement of the standard density estimate for
quasi-minimizers of the perimeter to a uniform estimate independent of
$\Omega$.
\begin{lemma}[Uniform density bound] \label{lem-something} %
  Let $\Omega \subset \mathbb R^n$ be a minimizer of \eqref{E} with $|\Omega| =
  m$. Then for every $x \in \overline \Ome^M$ we have for some $c = c(\alpha, n)
  > 0$,
  \begin{align}
    |\Omega \cap B_1(x)| \ \geq \ c \min \{ 1, m \}.
  \end{align}
\end{lemma}
\begin{proof}
  For given $r > 0$ and $x \in \overline \Omega^M$, define the sets
  $F_2^r := \Omega \cap B_r(x)$ and $F_1^r := \Omega \backslash
  B_r(x)$.  Since $|F_1^r| + |F_2^r| = m$ and $|F_2^r| \leq \omega_n
  r^n$, there exists $C > 0$ depending only on $\alpha$ and $n$ such
  that the assumption \eqref{om2eps1m} of Lemma \ref{lem-split} is
  satisfied for all $r \leq r_0 := C \min(1, m^{1/n})$. Since $\Omega$
  is a minimizer, \eqref{est-sigma} hence cannot be satisfied for any
  $r \leq r_0$. That is, for all $r \leq r_0$ we have
  \begin{align} \label{SigP} %
    \Sigma^r := P(F_1^r) + P(F_2^r) - P(\Omega) > \frac12 E(F_2^r) > \frac12
    P(F_2^r).
  \end{align}
  On the other hand, by \cite[Proposition 1]{AmbrosioCasellesEtal-2001} and
  \cite[Theorem 3.61]{AmbrosioFuscoEtal-Book}, we have $\Sigma^r = 2 \mathcal
  H^{n-1} (\partial^* F_1^r \cap \partial^* F_2^r)$. In fact, since all the
  points belonging to $\partial^* F_1^r$ and $\partial^* F_2^r$ are supported on
  $\partial B_r(x)$ and have density 1/2 by \cite[Theorem
  3.61]{AmbrosioFuscoEtal-Book}, we have $\partial^* F_1^r \cap \partial^* F_2^r
  = \mathring{\Omega}^M \cap \partial B_r(x)$. Together with \eqref{SigP} this
  yields
  \begin{align}
    \label{SigPP}
    2 \mathcal H^{n-1} \left( \mathring{\Omega}^M \cap \partial B_r(x)
    \right) > \frac12 \left( \mathcal H^{n-1} (\partial^* \Omega \cap
      B_r(x)) + \mathcal H^{n-1} ( \mathring{\Omega}^M \cap \partial
      B_r(x) \right).
  \end{align}
  We now rearrange terms in \eqref{SigPP} and apply the relative
  isoperimetric inequality to the right-hand side, noting that if
  $|\Omega \cap B_r(x)| \geq \frac12 \omega_n r^n$ for some $r < r_0$,
  the conclusion still holds. This results in
  \begin{align}
    \label{SigPPP}
    \mathcal H^{n-1} \left( \mathring{\Omega}^M \cap \partial B_r(x)
    \right) \geq c \left| \Omega \cap B_r(x) \right|^{n - 1 \over n},
  \end{align}
  for some $c > 0$ depending only on $n$. Finally, denoting $U(r) :=
  \left| \Omega \cap B_r(x) \right|$ and since by Fubini's Theorem
  $dU(r)/dr = \mathcal H^{n-1} \left( \mathring{\Omega}^M
    \cap \partial B_r(x) \right)$ for a.e. $r < r_0$ by co-area
  formula, we arrive at the differential inequality
  \begin{align}
    \label{dUdr}
    {d U(r) \over dr} \geq c U^{n - 1 \over n}(r) \qquad \text{for
      a.e. } r < r_0.
  \end{align}
  Together with the fact that since $x \in \overline{\Omega}^M$, we
  have $U(r) > 0$ for all $r > 0$, which implies that $U(r) \geq c
  r^n$ for some $c > 0$ depending only on $n$ for all $r \leq
  r_0$. The statement of the lemma then follows by choosing $r = r_0$.
\end{proof}

We next establish a basic regularity result for the potential $v$
defined in \eqref{v} for the range of $\alpha$ used in Theorem
\ref{thm-ball}.

\begin{lemma}
  \label{l:dvinfty}
  Let $F$ be a measurable set with $|F| \leq m$ for some $m > 0$. Then
  $\| v_F \|_{L^\infty(\mathbb R^n)} \leq C$ for some $C > 0$
  depending only on $\alpha$, $n$ and $m$. If, in addition, $\alpha
  \in (0, n - 1)$, then also $\| v_F \|_{W^{1,\infty}(\mathbb R^n)}
  \leq C'$ for some $C' > 0$ depending only on $\alpha$, $n$ and $m$.
\end{lemma}
\begin{proof}
  Boundedness of $v_F$ follows by the same argument as in the proof of
  Proposition \ref{thm-regularity}. Differentiating \eqref{v} in $x$,
  we obtain
  \begin{align}
    \label{dvinfty}
    |\nabla v_F(x)| \leq \alpha \int_F {1 \over | x - y |^{\alpha + 1}
    } \, d y \leq \alpha \int_{B_1(x)} {1 \over | x - y |^{\alpha + 1}
    } \, d y + \alpha |F| \leq C,
  \end{align}
  whenever $\alpha \in (0, n - 1)$.
\end{proof}

Let us point out that if $F$ also has a sufficiently smooth boundary,
then the potential $v_F$ may be estimated precisely near the
boundary. In particular, for the potential of the ball $v_B$
it is not difficult to see that for $r := |x| - 1$ and for $v_0 := v_{B}
\big|_{r=0}$ the leading order behavior of $v_B$ near the boundary is given by
\begin{align}
  \label{v-est} %
  v_0 - v^B(x) \ \sim \ %
  \begin{cases}
    r, &\text{if }
    \alpha < n-1, \\
    r \ln |r|, &\text{if } \alpha = n-1,  \\
    r^{n-\alpha}, & \text{if } \alpha > n-1.
  \end{cases}
\end{align} 
In particular, for $\alpha > n - 1$ the statement of Lemma
\ref{l:dvinfty} is false, even for a ball.

\medskip

The next lemma provides a tool
to compare the non-local part of the energy of two sets in terms of the
potential of one of the two sets:

\begin{lemma} \label{l:posdef} %
  Let $F$ and $G$ be measurable subsets of $\mathbb R^n$ with $|F| =
  |G| < \infty$. Let $v_F$ be the potential defined in \eqref{v}. Then
  for any $c \in R$, we have
    \begin{align}
      \label{VvFvG}
      V(F) - V(G) \ \leq \ 2 \bigg( \int_{F \backslash G} (v_F(x) - c) \, dx -
      \int_{G \backslash F} (v_F(x) - c) \, dx \bigg).
    \end{align}
  \end{lemma}

  \begin{proof}
    Let $\chi_F$ and $\chi_G$ be the characteristic functions of the
    sets $F$ and $G$, respectively. After some straightforward
    algebra one can write
    \begin{align}
      \label{FGalg}
      V(F) - V(G) \ & = \ \int_{\mathbb R^n} \int_{\mathbb R^n}
      {\chi_F(x) \chi_F(y) \over |x - y|^\alpha} \, dx \, dy -
      \int_{\mathbb R^n} \int_{\mathbb R^n}
      {\chi_G(x) \chi_G(y) \over |x - y|^\alpha} \, dx \, dy \notag \\
      & = \int_{\mathbb R^n} \int_{\mathbb R^n} {(\chi_F(x) +
        \chi_G(x)) (\chi_F(y) -
        \chi_G(y)) \over |x - y|^\alpha} \, dx \, dy \notag \\
      & = 2 \int_{\mathbb R^n} v_F(x) (\chi_F(x) - \chi_G(x)) \, dx
      \notag \\ & \qquad - \int_{\mathbb R^n} \int_{\mathbb R^n}
      {(\chi_F(x) - \chi_G(x)) (\chi_F(y) - \chi_G(y)) \over |x -
        y|^\alpha} \, dx \, dy.
    \end{align}
    In fact, since both $\chi_F$ and $\chi_G$ belong to $L^1(\mathbb
    R^n) \cap L^\infty(\mathbb R^n)$, one can use the Fourier
    transform to show that the last integral in \eqref{FGalg} is
    positive, see e.g. \cite{LiebLoss-2001}. Furthermore, since $|F| =
    |G|$, the right hand side of \eqref{FGalg} does not change if we
    replace $v_F(x)$ by $v_F(x) + c$ for arbitrary $c \in \R$. The
    assertion of the lemma follows.
  \end{proof}

\section{Existence}
\label{s-exist}

We prove existence of minimizers of $E$ with prescribed mass by suitably
localizing the minimizing sequence, which is possible for sufficiently small
mass, when the perimeter is the dominant term in the energy.

\begin{lemma}[Comparison with set of bounded
  support] \label{lem-bounded} %
  There is $m_{1} = m_{1}(\alp,n) > 0$ such that for every $m
  \leq m_{1}$ and every set of finite perimeter $F$ with $|F| = m$
  there exists a set of finite perimeter $G$ such that
  \begin{align}
    E(G) \leq E(F), %
    &&\text{and}&& %
    G \subset B_1.
  \end{align}
\end{lemma}
\begin{proof}
  Throughout the proof we will use the assumption $m \leq 1$.

  \medskip

  We may assume that $E(F) \leq E(B_r)$, with $|B_r| = m$, since
  otherwise we can choose $G = B_r$, which yields the assertion of the
  lemma. In particular,
  \begin{align}
    D(F) \lupref{isodef}= \frac{C}{r^{n-1}} \left(P(F) - P(B_r)
    \right) \ %
    \lupref{E}\leq \ \frac{C}{r^{n-1}} \left( V(B_r) - V(F) \right)
    \ %
    \leq \ C' r,
  \end{align}
  for some constants $C, C' > 0$ depending only on $\alpha$ and $n$, where we
  used Lemma \ref{l:dvinfty}. By the quantitative isoperimetric estimate
  \eqref{qiso} we therefore get the bound $\Delta(F,B_r) \leq C r^{1/2}$ for the
  Fraenkel asymmetry. Hence, after a suitable translation, we have $|B_r \Delta
  F| \leq C r^{n + \frac12}$. Since $|B_r \Delta F| = 2|F \backslash B_r|$, this
  in turn implies
  \begin{align} \label{qqiissoo} %
    |F \backslash B_r| \ \leq \ C r^{n + \frac12},
  \end{align}
  for some $C > 0$ depending only on $\alpha$ and $n$.

  \medskip

  For any $\rho > 0$, let $F_1 = |F \cap B_\rho|$ and $F_2 = |F \backslash
  B_\rho|$.  Note that by \eqref{qqiissoo} and for sufficiently small $m > 0$,
  the condition \eqref{om2eps1m} of Lemma \ref{lem-split} is satisfied for the
  two sets $F_1$ and $F_2$ for all $\rho > r$. We suppose that furthermore
  \begin{align} \label{est-sigma-2} %
    \Sigma \ := \ P(F_1) + P(F_2) - P(F) \ > \ \frac 12 E(F_2).
  \end{align}
  Indeed, if \eqref{est-sigma-2} is not satisfied, then Lemma \ref{lem-split}
  can be applied. Furthermore, both sets $\tilde F$ and $\hat F$ constructed in
  the proof of Lemma \ref{lem-split} are contained in a ball of radius $1$,
  which concludes the proof. Hence, we may assume in the following that
  \eqref{est-sigma-2} is satisfied for all $\rho > r$. In order to conclude the
  proof, we will apply an argument, similar to the one used in the proof of
  Lemma \ref{lem-something}. For this, we define the monotonically decreasing
  function $U(\rho) = |E \backslash B_\rho|$. Observe that by \eqref{qqiissoo},
  we have $U(\rho) \leq C\rho^{n+\frac12}$. Furthermore, as in the proof of
  Lemma \ref{lem-something}, we have
  \begin{align}
    {dU(\rho) \over d \rho} \ \leq \ - c U^{\frac {n-1}n}(\rho),
  \end{align}
  for some $c > 0$ depending only on $n$.  For $r$ sufficiently small, it then
  follows that $U(\rho) = 0$ for $\rho \geq 1$, which concludes the proof.
\end{proof}

\begin{proof}[Proof of Theorem \ref{thm-existence}]
  We choose a sequence of sets of finite perimeter $F_k$ with $|F_k| = m$ such
  that $E(F_k) \to \inf_{|F| = m} E(F)$. By Lemma \ref{lem-bounded}, we can
  choose a minimizing sequence such that these sets are uniformly bounded, i.e.
  $F_k \subset B_1(0)$. By lower semi-continuity of the perimeter, there is a
  set of finite perimeter $\Ome$ supported in $B_1(0)$ such that for some
  subsequence we have $\Delta(F_{k_j},\Omega) \to 0$ and $P(\Omega) \ \leq \
  \liminf_k P(F_k)$. Moreover, $|F_k| \to |\Omega|$ and, furthermore, by Lemma
  \ref{l:dvinfty} the nonlocal part of the energy convergences, i.e. $\NL(F_k)
  \to \NL(\Omega)$.  Hence, $|\Omega| = m$ and $E(\Omega) = \inf_{|F| = m}
  E(F)$, which concludes the proof.
\end{proof}

\section{Ball as the minimizer for small masses}
\label{s-ball}

In this section, we give the proof of Theorem \ref{thm-ball}.  For this it
is convenient to rescale length in such a way that the rescaled set $\Omega$ has
the mass $\omega_n$ of the unit ball.  We set $\lam = (\frac
m{\omega_n})^{1/n}$.
We also introduce a positive parameter $\eps > 0$ by
\begin{align} \label{def-eps} %
  \eps \ := \ \lam^{n+1-\alp} \ = \ \left( {m \over \omega_n}
  \right)^{\frac{n+1-\alp}{n}}.
\end{align}
Furthermore, we set $E_\eps(\Omega_\eps) := \lam^{n-1} E(\Omega)$
where $\Omega_\eps := \lambda^{-1} \Omega$. In the rescaled variables,
this yields the following energy to be minimized:
\begin{align}
  \label{Eeps}
  E_\eps(F) \ := \ P(F) + \eps \NL(F), \qquad |F| \ = \ \omega_n,
\end{align}

Note that by Theorem \ref{thm-existence}, the minimizers of $E_\eps$
exists for all $\eps \leq \eps_1$, where $\eps_1$ is related to $m_1$
via \eqref{def-eps}. Furthermore, the regularity result in Proposition
\ref{thm-regularity} holds for the minimizers of $E_\eps$. In this
section, however, we will need to further strengthen the statement of
Proposition \ref{thm-regularity} for the minimizers of $E_\eps$ to
allow for the regularity properties that are uniform in $\eps$. For
this we apply a result of Rigot \cite[Theorem 1.4.9]{Rigot-2000} which
we summarize in the following proposition:

  \begin{proposition}
    \label{p:regEeps}
    Let $3 \leq n \leq 7$. Then there exists $\eps_1 = \eps_1(\alpha,
    n) > 0$ such that for all $\eps \leq \eps_1$ there exists a
    minimizer $\Omega_\eps$ of $E_\eps$ in \eqref{Eeps}. Furthermore,
    the set $\Omega_\eps$ is open (up to a negligible set), and there
    exists $r_0 > 0$ depending only on $\alpha$ and $n$ such that if $
    x \in \partial \Omega_\eps$, then $\Omega_\eps \cap B_{r_0}(x)$ is
    (up to a rotation) the subgraph of a function of class
    $C^{1,\frac12}$, with the regularity constants depending only on
    $\alpha$ and $n$.
  \end{proposition}

Expressed in terms of the rescaled problem, Theorem \ref{thm-ball} takes the
form:
\begin{proposition} \label{prp-ball} %
  For all $3 \leq n \leq 7$ and for all $\alp \in (0,n-1)$ there is $\eps_0 =
  \eps_0(\alpha, n) > 0$ such that for all $\eps \leq \eps_0$ the unique (up to
  translation) minimizer of $E_\eps$ in \eqref{Eeps} is given by the unit ball.
\end{proposition}

We prove this proposition using of a sequence of lemmas.  As a first step, we
note that in the limit $\eps \to 0$, minimizers $\Omega_\eps$ of $E_\eps$
converge in $L^1$-norm sense (after a suitable translation) to the unit ball.
\begin{lemma} \label{lem-areaclose} %
  Let $\Omega_\eps$ be a minimizer for \eqref{Eeps}. Then there
  is $\eps_0 = \eps_0(\alpha, n) > 0$ such that for some $C =
  C(\alpha, n) > 0$ and all $\eps \leq \eps_0$, we have
  \begin{align} \label{est-areaclose} %
    \Delta(\Omega_\eps, B_1) \ \leq \ C \eps^{1/2}.
  \end{align}
\end{lemma}
\begin{proof} %
  The proof is based on an application of the quantitative isoperimetric
  inequality given by \eqref{qiso} \cite[Theorem 1.1]{FuscoMaggiPratelli-2008}.
  Since $\Omega_\eps$ is a minimizer, we have $E_{\eps}(\Omega_\eps) \leq
  E_{\eps}(B_1)$, which yields
  \begin{align} \label{area-1} %
    D(\Omega_\eps) %
    &\leq \frac\eps{n\omega_n} \Big( \int_{B_1(x_0)} \int_{B_1(x_0)}
    {1 \over |x - y|^\alpha} \ dx \, dy - \int_{\Omega_\eps}
    \int_{\Omega_\eps} {1 \over |x
      - y|^\alpha} \ dx \, dy \Big) \notag \\
    &\leq \ {\eps \over n} \sup_{x \in \mathbb R^n} \int_{B_1(0)} {1
      \over |x - y|^\alpha} \, dy + {\eps \over n} \sup_{x \in \mathbb
      R^n} \int_{\Omega_\eps} {1
      \over |x - y|^\alpha} \, dy   \notag \\
    & \leq \ {2 \eps \over n} \Big( \omega_n + \int_{B_1(0)} {1 \over
      |y|^\alpha} \, dy \Big) \leq C \eps,
  \end{align}
  where in the last line we split the integration over $\mathbb R^n
  \backslash B_1(x)$ and $B_1(x)$, respectively. Together with
  \eqref{qiso}, this concludes the proof.
\end{proof}

In fact, using the regularity result in Proposition \ref{p:regEeps}, we can
  show that for sufficiently small $\eps$ every minimizer $\Omega_\eps$ may be
  represented by the subgraph of a map $\rho: \partial B_1(0) \to \partial \Omega$,
  with $\partial \Omega_\eps$ close to $\partial B_1(0)$ in the $C^1$ norm on
  $\partial B_1(0)$: 
\begin{lemma} \label{lem-close} %
  For all $3 \leq n \leq 7$, $\alp \in (0,n)$ and $\delta > 0$ there is $\eps_0
  = \eps_0(\alpha, n, \delta) > 0$ such that for all $\eps \leq \eps_0$ every
  minimizer $\Ome_\eps$ of \eqref{Eeps} is given by (up to a negligible set)
  \begin{align} \label{global-graph} %
    \Omega_\eps - x_0 \ = \ \{ x \ : \ |x| < \ 1 + \rho(x/|x|) \} && %
    \text{for some $\rho \in C^{1,\frac12}(\partial B_1(0))$},
  \end{align}
  with $\| \rho \|_{W^{1,\infty}(\partial B_1(0))} \leq \delta$ and
  $x_0 \in \mathbb R^n$ the barycenter of $\Omega$.
\end{lemma}
\begin{proof}
  By Proposition \ref{p:regEeps}, for every ball of radius $r < r_0$
  and every $x \in \partial \Omega_\eps$ the set $\partial \Omega_\eps
  \cap B_r(x)$ approaches the tangent hyperplane $\Pi_r(x)$ at $x$ as
  $r \to 0$ in the $C^1$ sense, with moduli of continuity depending
  only on $\alpha$ and $n$. Similarly, for any $x_1$ fixed the set
  $\partial B_1(x_1) \cap B_r(x)$ is either empty or is close in the
  $C^1$ sense to a hyperplane $\tilde \Pi_r(x)$ for all sufficiently
  small $r > 0$. Therefore, for every $c > 0$ sufficiently small
  depending only on $\alpha$ and $n$ there exist $r \in (0, r_0)$
  depending only on $\alpha$, $n$ and $c$ such that if $B_1(x_1)$ is a
  ball that minimizes the Fraenkel asymmetry $\Delta(\Omega_\eps,
  B_1)$, we have $\Delta(\Omega_\eps, B_1) \, \omega_n = |\Omega_\eps
  \triangle B_1(x_1)| \geq | (\Omega_\eps \triangle B_1(x_1)) \cap
  B_r(x)| \geq c r^n$, unless $\partial B_1(x_1) \cap B_r(x) \not=
  \varnothing$ and $\Pi_r(x) - x$ is sufficiently close to $(\tilde
  \Pi_r(x) - x)/r$ in $B_1(0)$ in the Hausdorff sense (with closeness
  controlled by $c$). Closeness of $\partial \Omega_\eps$ and
  $\partial B_1(x_1)$ in the $C^1$ sense controlled by $\eps$ then
  follows by Lemma \ref{lem-areaclose} for all $\eps \leq \eps_0$,
  with $\eps_0 > 0$ depending only on $\alpha$ and $n$.

  \medskip

  Thus, locally $\partial \Omega_\eps$ may be represented by a $C^1$
  map from an open subset of $\partial B_1(0)$ to $\mathbb R^n$. In
  fact, by the uniform $C^1$ closeness of $\partial\Omega_\eps$ to
  $\partial B_1(x_1)$ this map can be extended from the neighborhood
  of each point $x \in \partial \Omega_\eps$ to a global $C^1$ map
  from $\partial B_1(0)$ to $\partial \Omega_\eps$. This implies that
  $\partial \Omega_\eps$ may be represented by a union of finitely
  many connected components consisting of non-intersecting graphs of
  $C^1$ maps from $\partial B_1(0)$ to $\mathbb R^n$, with the degree
  of closeness controlled by $\eps$ and with the perimeter of each
  component approaching $P(B_1(x_1))$ as $\eps \to 0$. Since by the
  minimizing property of $\Omega_\eps$ and positivity of the non-local
  term we have $P(\Omega_\eps) \leq E_\eps(\Omega_\eps) \leq
  E_\eps(B_1(x_1)) \leq P(B_1(x_1)) + C \eps$ for some $C > 0$, we
  conclude that for all $\eps \leq \eps_0'$ with $\eps_0' > 0$
  depending only on $\alpha$ and $n$ the set $\partial \Omega_\eps$
  consists of only one connected component and, therefore,
  $\Omega_\eps$ can be represented by \eqref{global-graph}.

\medskip

Finally, since $\Omega_\eps$ is a simply connected open set whose
boundary $\partial \Omega_\eps$ is close in $C^1$ sense to $\partial
B_1(x_1)$, the quantity $|x_0 - x_1|$, where $x_0$ is the barycenter
of $\Omega_\eps$, is small and controlled by $\eps$ as well. The
statement of the lemma then follows from the $C^1$ closeness of
$B_1(x_0)$ to $B_1(x_1)$.
\end{proof}

We next use a bound on the isoperimetric deficit of almost spherical sets
derived by Fuglede \cite{Fuglede-1989}:
\begin{lemma}
  For all $3 \leq n \leq 7$, all $\alpha \in (0, n)$ and all
    $\eps \leq \eps_0$, where $\eps_0 = \eps_0(\alpha, n)$, the
  minimizer $\Omega_\eps$ of \eqref{Eeps} satisfies
  \begin{align} 
    \label{delD} 
    \| \rho \|^2_{L^2(\partial B_1(0))} + \| \nabla \rho
    \|^2_{L^2(\partial B_1(0))} \ \leq \ C D(\Omega_\eps),
  \end{align}
  where $\rho$ is as in Lemma \ref{lem-close}, for some universal $C >
  0$.
\end{lemma}
\begin{proof}
  Since $\Omega_\eps$ is near the ball in $C^1$-norm, choosing $\delta > 0$ in
  Lemma \ref{lem-close} sufficiently small we can apply the result by Fuglede in
  \cite[Theorem 1.2]{Fuglede-1989} to yield the estimate.
\end{proof}

\begin{proof}[Proof of Proposition \ref{prp-ball}]
  By Proposition \ref{p:regEeps}, there exists a minimizer
  $\Omega_\eps$ of $E_\eps$ if $\eps$ is sufficiently small.
  Furthermore, the set $\Omega_\eps$ satisfies the conclusions of
  Lemma \ref{lem-close}. Since $\Omega_\eps$ is a minimizer, we have
  $E_\eps(\Omega_\eps) \leq E_\eps(B_1(x_0))$, where $x_0$ is the
  barycenter of $\Omega_\eps$, which is equivalent to
  \begin{align} \label{dOmballal11} %
    D(\Omega) \leq \frac\eps{n\omega_n} \left( V(B_1(x_0)) - V(\Omega) \right).
  \end{align}
  On the other hand, choosing $c = v_0$ in Lemma \ref{l:posdef}, where $v_0$ is
  as in \eqref{v-est}, and applying Lemma \ref{l:dvinfty}, we obtain
  \begin{align} \label{Enlal1} %
    \hspace{2ex} & \hspace{-2ex} V(B_1(x_0)) - V(\Omega) \notag \\
    &\upref{VvFvG}\leq 2 \left( \int_{B_1(x_0) \backslash \Omega}
      (v^B(x - x_0) - v_0) \, dx - \int_{\Omega \backslash
        B_1(x_0) } (v^B(x - x_0) - v_0) \, dx  \right) \notag \\
    & \leq 2 \int_{\Omega \triangle B_1(x_0)} |v^B(x - x_0) - v_0| \, dx \notag \\
    & \leq \ C \int_{\partial B_1(x_0)} \left( \int_0^{\rho(x)} t \,
      dt \right) d \mathcal H^{n-1}(x) \ %
    \leq \ C' \nltL{\rho}{\partial B_1(x_0)}^2
  \end{align}
  for some $C, C' > 0$ depending on $\alpha$ and $n$.
  Combining this inequality with \eqref{delD},
  \eqref{dOmballal11} and \eqref{Enlal1}, we get
  \begin{align}
    \label{Dal1}
    c \nlt{\rho}^2 \ \leq D(\Omega) \ \leq \ C \eps \nlt{\rho}^2
  \end{align}
  for some universal $c > 0$ and some $C > 0$ depending only $\alpha$
  and $n$. Therefore, as long as $\eps$ is small enough, we have
  $D(\Omega) = 0$. This implies that $\Omega = B_1(x_0)$, thus
  concluding the proof of the proposition.
\end{proof}

 \section{Non-existence and equipartition of energy}
\label{s-non}  

We now establish Theorem \ref{thm-scaling}.  The following interpolation
estimate is a generalization of a corresponding two-dimensional result, proved
in \cite[Eq. (5.3)]{KnuepferMuratov-2011b}:
\begin{lemma}[Interpolation] \label{lem-interpolation} %
  For any $u \in BV(\R^n) \cap L^\infty(\R^n)$, we have
  \begin{align} \label{int} %
    \int_{\R^n} u^2 \, dx \ \leq \ C
    \ncoL{u}{\R^n}^{\frac{n-\alpha}{n+1-\alpha}} \left(
      \int_{\R^n}|\nabla u| \, dx
    \right)^{\frac{n-\alpha}{n+1-\alpha}} \left( \int_{\R^n}
      \int_{\R^n} \frac{u(x)u(y)}{|x-y|^\alpha} \ dx dy
    \right)^{\frac{1}{n+1-\alpha}},
  \end{align}
  for some $C > 0$ depending only on $\alpha$ and $n$.
\end{lemma}
\begin{proof}
  Follows by a straightforward extension of the proof in
  \cite[(5.3)]{KnuepferMuratov-2011b}.
\end{proof}

Using the result in Lemma \ref{lem-interpolation}, we are ready to
give the proof of Theorem \ref{thm-scaling}.

\begin{proof}[Proof of Theorem \ref{thm-scaling}]
  The lower bound is a consequence of Lemma
  \ref{lem-interpolation}. Indeed, for any set of finite perimeter $F$
  with $|\Omega| = m$, an application of \eqref{int} with $u$ as the
  characteristic function of $F$ yields
  \begin{align} \label{lbest7} %
    m \ \leq \ C P^{\frac{n-\alpha}{n+1-\alpha}}(F)
    \NL^{\frac{1}{n+1-\alpha}}(F) \ %
    \leq \ C E(F).
  \end{align}
  for some constant $C > 0$ depending only on $\alpha$ and $n$; the lower bound
  follows. Since the estimate \eqref{int} is multiplicative, the assertion on
  equipartition of energy in \eqref{eq:1} follows as well. The proof of the
  upper bound can be shown by explicit construction of a collection of balls
  sufficiently far apart. The argument proceeds similarly as in \cite[Lemma
  5.1]{KnuepferMuratov-2011b}.
\end{proof}

We now turn to the proof of Theorem \ref{thm-nonexistence}. We proceed
by first establishing a lemma about the spatial extent of minimizers
for large masses.

\begin{lemma} \label{lem-diam} %
  Let $\Omega$ be a minimizer of $E$ with $|\Omega| = m$ and $m \geq 1$. Then
    \begin{align}
      \label{dmM}
      c m^{\frac 1\alpha} \leq \diam \overline{\Omega}^M \leq C m,
    \end{align}
    for some $C, c > 0$ depending only on $\alpha$ and $n$.
\end{lemma}
\begin{proof}
  We first recall that by Lemma \ref{lem-indec} the minimizer $\Omega$ is
  essentially bounded and indecomposable; in particular $d := \diam \overline
  {\Omega}^M < \infty$. By Theorem \ref{thm-scaling} and in view of \eqref{v},
  we get for some $C > 0$ depending only on $\alpha$ and $n$ that
  \begin{align}
    \frac{m^2}{d^\alpha} \ \leq \ V(\Omega) \leq E(\Omega) \ \leq \
    Cm,
  \end{align}
  which implies the lower bound in \eqref{dmM}.

  \medskip

  We next turn to the proof of the upper bound in \eqref{dmM}.
  Clearly, we may assume that $d > 5$. Furthermore, without loss of
  generality we may assume that there exist $x^{(1)}, x^{(2)} \in
  \overline{\Omega}^M$ such that $x^{(1)} \cdot e_1 < 1$ and $x^{(2)}
  \cdot e_1 > d - 1$. Let $N$ be the largest integer smaller than
  $\frac{d-2}{3}$.  Since $\Omega$ is indecomposable, there exist $N$
  disjoint balls $B_1(x_j)$, $j = 1, \ldots, N$ with $x_j \in
  \overline{\Omega}^M$ such that $3 j - 1 < x_j \cdot e_1 < 3 j$. The
  upper bound now follows from an application of Lemma
  \ref{lem-something}:
  \begin{align}
    m \ = \ |\Omega| \ \geq \ \sum_{j=1}^N |B_1(x_j) \cap \Omega| \ \geq \ cN \
    \geq \ c' d,
  \end{align}
  for some universal $c' > 0$.
\end{proof}

As a direct consequence of Lemma \ref{lem-diam}, we get
\begin{corollary}
  \label{cor-non1}
  The statement of Theorem \ref{thm-nonexistence} holds if $\alpha <
  1$.
\end{corollary}
It remains to give the proof of Theorem \ref{thm-nonexistence} in the
case $\alpha \in [1,2)$.
\begin{proof}[Proof of Theorem \ref{thm-nonexistence}] 
  By Corollary \ref{cor-non1}, we may assume that $\alpha \in
  [1,2)$. Arguing by contradiction, we assume that for every $m_2 > 0$
  there exists a minimizer $\Omega$ of $E$ with $|\Omega| = m$ for
  some $m > m_2$. We define $d := \diam \overline{\Omega}^M < \infty$
  and 
    \begin{align}
      \label{Ut}
      U(t) := |\Omega \cap \{ x \in \mathbb R^n : 0 < x \cdot e_1
    < t \}| \qquad \forall t > 0.
    \end{align}
    According to Lemma \ref{lem-diam}, we may assume that $4 < d < \infty$. By a
    suitable rotation and translation of the coordinate system, we may further
    assume that there are points $x^{(1)}$, $x^{(2)} \in \overline{\Omega}^M$
    with $x^{(1)} \cdot e_1 < 1$ and $x^{(2)} \cdot e_1 > d - 1$, and
    \begin{align}
      \label{UdUd2}
      U(d) = m \qquad \text{and} \qquad U(\tfrac12 d) \leq \tfrac12 m.
    \end{align}
    Let $\Pi_t$ be the hyperplane $\Pi_t : = \{ x \in \R^n: \ x \cdot e_1 = t
    \}$.
    For a given $t \in (0, \tfrac12 d)$, we cut the set $\Omega$ by the
    hyperplane $\Pi_t$ into two pieces, which are then moved apart to a large
    distance $R > 0$, with the new set denoted as $\Omega_t^R$. In general, the
    perimeter of $\Omega^R_t$ is larger than that of $\Omega$, while the
    non-local part of the energy of the new set is smaller. To make this
    statement quantitative, as in the proof of Lemma \ref{lem-something} we
    define
    \begin{align} \label{nex-rho} %
      \rho(t) \ := \ \frac12 \left( P(\Omega_t^R) - P(\Omega) \right).
  \end{align}
  Using the same arguments as in the proof of Lemma
  \ref{lem-something}, we have $\rho(t) =
  \HH^{n-1}(\mathring{\Omega}^M \cap \Pi_t)$. Hence, by Fubini's
  Theorem we obtain
  \begin{align}
    \label{PPrho}
    \NL(\Omega_t^R) - \NL(\Omega) \ %
    &\leq \ - \frac m{2 d^\alpha} U(t) +
    K(R) \notag \\
    &= \ - \frac m{2 d^\alpha} \int_0^t \rho(t') \, dt' + K(R),
  \end{align}
  where $K(R) \to 0$ as $R \to \infty$. Combining \eqref{nex-rho} and
  \eqref{PPrho}, we obtain
  \begin{align}
    E(\Omega_t^R) - E(\Omega) \ %
    \ \leq \ 2\rho(t) - \frac m{2 d^\alpha} \int_0^t \rho(t') \,
    dt' + K(R).
  \end{align}
  Since $\Omega$ is assumed to be a minimizer, we have in particular
  \begin{align} \label{sode} %
    2 \rho(t) \ \geq \ \frac m{2 d^\alpha} \int_0^t \rho(t') \, dt' && %
    \forall t \in (\tfrac14 d, \tfrac12 d).
  \end{align}
  From \eqref{sode} and the fact that $U \in C^{0,1}([0,d])$ we infer that
  \begin{align} \label{eq:18} %
    {d U(t) \over dt} \ \geq \ \frac m{4 d^\alpha} U(t) && %
    \text{for a.e. $t \in (\tfrac14 d, \tfrac12 d)$},
  \end{align}
  Integrating this expression yields
  \begin{align} \label{eq:19} %
    U(t) \ \geq \ U(\tfrac 14 d) e^{m ( t - \tfrac 14 d
      ) / (4 d^\alpha)}.
  \end{align}
  Note that by Lemma \ref{lem-something} we have $U \left( \tfrac14 d
  \right) \geq |\Omega \cap B_1(x^{(1)})| \geq c > 0$ which in view of
  \eqref{eq:19} implies $U(\tfrac12 d) \ \geq \ C e^{\frac{1}{16} m
    d^{1-\alpha}}$. This contradicts the inequality in \eqref{UdUd2},
  thus concluding the proof.
\end{proof}

\section*{Acknowledgements}

The authors would like to acknowledge valuable discussions with
R. V. Kohn, M. Novaga and S. Serfaty. C. B. M. was supported, in part,
by NSF via grants DMS-0718027 and DMS-0908279. H. K. would like to
thank the program SFB 611 ``Singular Phenomena and Scaling in
Mathematical Models'' for support.

\bibliographystyle{plain}

\bibliography{all}

\end{document}